\documentclass[reqno,11 pt]{amsart}
\usepackage{graphicx} % Required for inserting images
\usepackage{amsmath,amssymb,latexsym,float}
\usepackage{amsthm}
\usepackage{color}  
\usepackage[all]{xy}
\usepackage{tikz-cd}
\usepackage{accents} 
\usepackage{lmodern}
\usepackage{multirow}
\usepackage{multicol}
\usepackage{longtable}
\usepackage{caption}
\usepackage{cite}
\usepackage{enumitem}
\usepackage{listings}
\usepackage{xcolor}
\usepackage{placeins}

\usepackage[colorlinks=true,
            linkcolor=orange,
            citecolor=cyan,
            urlcolor=blue]{hyperref}
\usetikzlibrary{cd}
\usepackage[left=0.25in,right=0.25in, top= -0.5in, bottom= -0.5 in]{geometry}
\usepackage{array}
\usepackage{changepage}
\numberwithin{equation}{section}

\newcommand{\F}{{\mathbb{F}}}
\newcommand{\Z}{{\mathbb{Z}}}
\newcommand{\Q}{{\mathbb{Q}}} 
\newcommand{\R}{{\mathbb{R}}} 
\newcommand{\C}{{\mathbb{C}}} 
\newcommand{\A}{{\mathbb{A}}}

\newcommand{\p}{{\mathfrak{p}}}
\newcommand{\OO}{{\mathcal{O}}}
\newcommand{\tors}{{\mathrm{tors}}}

\newcommand{\Gal}{\mathrm{Gal}}

\newcommand{\rk}{\mathrm{rank}}

\newcommand{\tr}{\mathrm{Tr}}
\newcommand{\gl}{\mathrm{GL}}

\newcommand{\la}{{\langle}}
\newcommand{\ra}{{\rangle}}
\newcommand{\sel}{\mathrm{Sel}}
\newcommand{\fr}{{\mathrm{Frob_p}}}

\makeatletter
\newcommand*{\rom}[1]{\expandafter\@slowromancap\romannumeral #1@}
\makeatother

\DeclareFontFamily{U}{wncy}{}
    \DeclareFontShape{U}{wncy}{m}{n}{<->wncyr10}{}
    \DeclareSymbolFont{mcy}{U}{wncy}{m}{n}
    \DeclareMathSymbol{\Sh}{\mathord}{mcy}{"58}

\lstdefinelanguage{Sage}{
  language=Python,
  morekeywords={var, assume, simplify, factor, show, expand, diff, integrate, solve},
}

\theoremstyle{plain}
\newtheorem{theorem}{Theorem}[section]
\newtheorem*{theorem*}{Theorem}
\newtheorem{proposition}[theorem]{Proposition}
\newtheorem{rem}[theorem]{Remark}
\newtheorem{lemma}[theorem]{Lemma}
\newtheorem{corollary}[theorem]{Corollary}

\title{Explicit  mock Heegner points and BSD formula on certain Mordell curves}
\author{Sumita Gunri, Somnath Jha,  Dipramit Majumdar}
\address{\begin{scriptsize}Sumita Gunri, \href{mailto:sumitagunri111@gmail.com}{sumitagunri111@gmail.com}, Department of Mathematics, IIT Madras, Chennai 600036, INDIA.\end{scriptsize}}
\address{\begin{scriptsize}Somnath Jha, \href{mailto:jhasom@gmail.com}{jhasom@gmail.com}, Department of Mathematics, IIT Kanpur, Kanpur 208016, India\end{scriptsize}}
\address{\begin{scriptsize}Dipramit Majumdar, \href{mailto:dipramit@gmail.com}{dipramit@gmail.com}, Department of Mathematics, IIT Madras, Chennai 600036, INDIA,  and Institut für Mathematik, Universität Heidelberg, GERMANY.\end{scriptsize}}

\begin{document}
\begin{abstract}
   For a natural number $a$, let $E_{2a}$ be the Mordell elliptic curve  $X^3 +Y^3=2a$.  We give an explicit construction of (mock) Heegner point on the Mordell curve $E_{2p}$ for a prime  $p\equiv 4 \mod 9$ and $E_{2p^2}$ for a prime $p\equiv 7 \mod 9$, under the assumption that $2$ is not a cube modulo $p$. We also verify the explicit Gross-Zagier formula for these curves and go on to show that the BSD formula holds for these curves up to a $2$-adic unit. Using a result of Burungale-Flach, we show that the full BSD formula holds for the rank zero curve $E_{2p}$ for $p\equiv 7 \mod 9$ and $E_{2p^2}$ for $p\equiv 4 \mod 9$, whenever $2$ is not a cube modulo $p$.
\end{abstract}

\maketitle

\vspace{-1cm}
\section*{Introduction}

Solutions to many  Diophantine problems are intricately related to rational points on CM elliptic curves. Let us call an integer $n$ a rational cube sum if there exist two rational numbers $a$ and $b$ such that $n=a^3+b^3$. The study of integers $n$ that are rational cube sums has a rich history and can be traced back to the classical works of Sylvester \cite{sylvester}, Selmer \cite{selmer}, Satg\'e \cite{satge}, {Lieman \cite{lieman}} and up to a very recent work of Alp{\"o}ge-Bhargava-Shnidman-Burungale-Skinner \cite{alpoge-bhargava-shnidman}. Without any loss of generality, we may assume that $n$ is cube-free and greater than $2$. Then the elliptic curve $Y^3+X^3=nZ^3$, expressed in Weierstrass form  as $E_n:Y^2=X^3-432n^2$, has $E_n(\Q)_{\tors}=0$ and hence $n$ is a rational cube sum $\Leftrightarrow \text{rank}_\Z \ E_n(\Q) > 0$. Let $L(s,E_n)$ be the complex $L$-function of $E_n/\Q$. Then recall that the rank part of the Birch and Swinnerton-Dyer (BSD) conjecture predicts that $\text{rank}_\Z \ E_n(\Q) =\mathrm{ord}_{s=1} L(s, E_n)$. 

If we restrict ourselves to prime numbers, then by a classical computation of Birch-Stephens, the global root number $\omega(E_p/\Q):=(-1)^{\text{ord}_{s=1} L(s, \,E_p)}$,  the sign of the functional equation $L(s, E_p)$  is $-1$ for $p\equiv 4,7,8 \mod 9$ and  is $+1$  for $p\equiv 1,2,5 \mod 9$ and hence the expressibility of primes as rational cube sums in the congruence classes modulo $9$ has also been studied in great detail in the literature. There is a conjecture, typically
attributed to Sylvester (see \cite{mock heegner}), which predicts that  primes  $p \equiv
4,7,8 \mod 9$ are rational cube sums. 
Dasgupta-Voight \cite{mock heegner} showed that for  primes $p \equiv 4,7
\mod 9$, both $p$ and $p^2$ are rational cube sums provided $3$ is
not a cube modulo $p$. The case $ 8 \mod 9$ seems more subtle and there are partial results by Yin \cite{yin} (also see \cite{majumdar-sury, jha-majumdar-sury}). For a composite integer $n$ with not too many prime factors, cube sum problem has also been considered, although the results are far more scarce. In a classical work, Satg\'e \cite{satge} studied the cube sum problem for $2p$ (for odd prime $p \equiv 2 \mod 9$) and $2p^2$ (for prime $p \equiv 5 \mod 9$). He used Weber functions and his proof can be reinterpreted as mock Heegner point construction on $E_{2p}$ (see \cite{survey article}). Such methods have also been employed by \cite{coward, cube sum, shu-yin} and others. In this article, we consider the case $n=2p, 2p^2$, where $p$ is a prime with $p \equiv 4,7 \mod 9$. For ease of notation,  set
\begin{equation}\label{qdef}
q = \begin{cases} p & \text{ if } p \equiv 4 \mod 9, \\
p^2 & \text{ if } p \equiv 7 \mod 9.
\end{cases}\end{equation}
The global root number of $E_{2q}/\Q$, $\omega(E_{2q}/\Q)=-1$. On the other hand, using $3$-descent, it was shown in \cite{majumdar-shingavekar, cube sum jnt} that the $3$-Selmer group of $E_{2q}$ over $\Q$, $\sel^3(E_{2q}/\Q)$ satisfies $\dim_{\F_3} \ \sel^3(E_{2q}/\Q) \\=1$, so that $\mathrm{rank}_\Z \  E_{2q}(\Q) \leq 1$. Hence the BSD (parity) conjecture predicts that $\mathrm{rank}_\Z \ E_{2q}(\Q) =1$. In this paper, under the assumption that $2$ is not a cube modulo $p$, we show that $\mathrm{rank}_\Z \ E_{2q}(\Q) =1$ and in particular, $2q$ is a cube sum. 

More precisely, we start with an explicit mock Heegner point $P_1$ on the elliptic curve $E:y^2=x^3+1$ defined over the ring class field $H_{6p}$ of conductor $6p$ of $K=\Q(\sqrt{-3})$. We largely follow the framework used in the existing literature for the construction of the candidate Heegner point, however the proof that $S_1 = \tr_{H_{6p}/K(\sqrt[3]{p})} P_1$ is a nontorsion point in $E(K(\sqrt[3]{p}))$ is different from the earlier works. The novelty lies in constructing another explicit torsion point $P' \in E(H_{6p})$ that is independent of $p$, such that the reduction of $P_1$ modulo $p$ is the image under the Frobenius map of the reduction of $P'$ modulo $p$. The argument used to show $P_1$ is nontorsion (Theorem \ref{nontorsion}), requires that $2$ is not a cube modulo $p$. Using the cubic twist isomorphism, we are then able to construct a nontorsion point in $E_{2q}(K)$. Since $\mathrm{rank}_\Z \  E_{2q}(K) = 2 \, \mathrm{rank}_\Z \  E_{2q}(\Q)$, it follows that $\mathrm{rank}_\Z \ E_{2q}(\Q) >0$ and hence $\mathrm{rank}_\Z \ E_{2q}(\Q)=1$. The construction of the nontorsion point using mock Heegner point allows us to use the `variation of the Gross-Zagier formula' developed in \cite{G-Z formula}. Let  $\chi: \Gal({K^\mathrm{ab}}/K) \to \C^\times$ be the cubic character given by  $\chi(\tau): = (\sqrt[3]{p})^{\tau -1}$. The base change $L$-function $L(s, E, \chi)$ over $K$  has global root number $-1$ and factors as $$L(s, E, \chi)=L(s, E_{2q})L(s, E_{2q^2}).$$ The explicit Gross-Zagier formula allows us to relate the derivative $L'(1,E, \chi)$ with the height of $S_1$ and since $S_1$ is nontorsion, it follows that $\mathrm{ord}_{s=1} L(s,E_{2q})=1$. Consequently, by works of Gross-Zagier \cite{gross-zagier}, Kolyvagin \cite{kolyvagin} and Rubin \cite{rubin}, the Tate-Shafarevich group of $E_{2q}$ over $\Q$, $\Sh(E_{2q}/\Q)$ is a finite group. Finally, we compute the index of the Heegner point to prove: 
\begin{theorem}\label{main}
    Let $p \equiv 4,7 \mod 9$ be a prime such that $2$ is not a cube in $\F_p$. Then the elliptic curve $E_{2q}: y^2=x^3-27q^2$ satisfies the rank part of the BSD conjecture, that is
    $$\mathrm{ord}_{s=1} L(E_{2q}, 1) = \mathrm{rank}_\Z \ E_{2q}(\Q)=1.$$
   In particular, $2q$ is a rational cube sum.  Moreover, the order of the finite group $ \Sh(E_{2q}/\Q)$ is
    $$ |\Sh(E_{2q}/\Q)|=  \frac{L'(E_{2q},1) |E_{2q}(\Q)_{\tors}|^2}{\Omega_{E_{2q}} \mathrm{Reg}(E_{2q}) \prod_{v} c_{v}(E_{2q})} u, \qquad u \in \Z\Big[\frac{1}{2}\Big]^\times     $$
    as predicted by the BSD conjecture, up to a unit in $\Z\big[\frac{1}{2}\big]$. \qed
\end{theorem}

 If $2$ is not a cube in $\F_p$, then  for primes $p\equiv 4,7 \mod 9$, we have  $ \mathrm{rank}_\Z \ E_{2q^2}(\Q)=0$ (see \cite{majumdar-shingavekar, cube sum jnt}). As a consequence of the explicit Gross-Zagier formula associated to the constructed mock Heegner point $S_1$, we strengthen the above result using \cite{burungale-flach} and establish the following:

\begin{theorem}\label{mainthm2}
    Let $p \equiv 4,7 \mod 9$ be a prime such that $2$ is not a cube in $\F_p$. Then the rank zero elliptic curve $E_{2q^2}: y^2=x^3-27q^4$ satisfies the full BSD conjecture. \qed
\end{theorem}
 \begin{rem} The elliptic curves $E_{2q^2}$ have global root number $\omega(E_{2q^2}/\Q) = +1$. When ordered by the absolute value, it is expected that for $100\%$ of $p$, $\mathrm{rank}_{\Z}\ E_{2q^2}(\Q) =0$, but there are infinitely many $p$ for which $\mathrm{rank}_{\Z}\ E_{2q^2}(\Q) >0$. Indeed, if $p \equiv 7 \mod{9}$, it follows from \cite{jha-majumdar-sury} that there are infinitely many $p$ for which $\mathrm{rank}_{\Z}\ E_{2q^2}(\Q) >0$, for example,  $\mathrm{rank}_\Z \ E_{2.(43^2)^2}(\Q) =2$.
 The hypothesis $2$ is not a cube modulo $p$ ensures that $\mathrm{rank}_{\Z}\ E_{2q^2}(\Q) =0$.
 \end{rem}
 \noindent {\bf Structure of the Article}
After the introduction, the article has $3$ sections. In Section \ref{section 2}, we discuss some basics of the modular curve $X_0(36)$ and give the construction of the Heegner point with relevant Galois action. Next, in Section \ref{sec:nontorsion}, we prove, using reduction mod $p$, that the constructed Heegner point is indeed nontorsion. Finally, in Section \ref{section 3} we compute the explicit Gross-Zagier formula and also verify the BSD conjecture for both rank zero and rank one curves.
 
\noindent {\bf Acknowledgment:} The third author thanks  Otmar Venjakov for hosting him at  Universit\"at of Heidelberg as a visiting researcher. Sumita Gunri is supported by PhD  fellowship at IIT Madras.

\section{\texorpdfstring{The modular curve $X_0(36)$ and Galois action on Heegner point}{The modular curve X0(36) and Galois action on Heegner point}}\label{section 2}

Let $n  $ be a cube-free positive integer with $n \neq 1, 2, 4$. Let  
$K=\Q(\zeta),  L= K(\sqrt[3]{n})$ where $\zeta=\frac{-1+\sqrt{-3}}{2}$ is a primitive cube root of unity in $\C$. Then $L/K$ is a cyclic cubic extension with $\Gal(L/K)= \la \tau \ra$, where the generator $\tau$ is determined by  $(\sqrt[3]{n})^{\tau -1}= \zeta$. Let $E/\Q$ be the elliptic curve given by the Weierstrass equation $E:=E_0: y^2 = x^3+1$ endowed with CM by $K$ and we fix an isomorphism $[.]: \OO_K\rightarrow \mathrm{End}_\C (E)$ by $[\zeta](x,y)=(\zeta x, y).$ We have $\rk_\Z \ E(K)=0$. Consider the cubic twists $E_1^{(n)}: y^2 = x^3+n^2$ and $E_2^{(n)}: y^2 = x^3+n^4$ of $E$ associated with $L/K$. For an abelian group $A$, $A_{\tors}$ denotes the torsion subgroup of $A$. For $n\in \mathbb{N}$, $A[n]=\{x\in A\mid nx=0\}$ and $A[p^\infty]=\cup_{n\geq 1} A[p^n]$.

\begin{proposition}

    We have $E(L)_{\tors} = E(K)_{\tors} \simeq \Z/6\Z \times \Z/2\Z$.
    Explicitly, $E(L)_{\tors}=$
    $$
    \left\{ \mathcal{O},
 (2 , 3 ),
 (0 , 1),
 (-1 , 0 ),
 (0 , -1),
 (2 , -3 ),
 (-\zeta , 0 ),
 (2\zeta^2 , 3 ),
 (2 \zeta , -3 ),
 (-\zeta^2 , 0 ),
 (2\zeta , 3 ),
 (2\zeta^2 , -3) \right\}.$$
\end{proposition}
\begin{proof}
 We prove that for every prime $p$, $E(L)_{\tors}[p^\infty] = E(K)_{\tors}[p^\infty]$.\\
  Let $f_0:E(K) \xrightarrow{\sim} E(L)^{\tau = 1} \hookrightarrow E(L) $ be the natural embedding.
 For $i=1,2$, by general theory of twisting we have isomorphisms $\phi_i: E(L)^{\tau = \zeta^i} := \{ X \in E(L) \mid  X^\tau= \zeta^i X \} \xrightarrow{\sim} E_1^{(n^i)}(K) \simeq E_i^{(n)}(K)$, and hence we obtain maps $f_i:= \phi_i^{-1}: E_i^{(n)}(K) \xrightarrow{\sim} E(L)^{\tau = \zeta^i} \hookrightarrow E(L)$. Now define
    $f : E_0(K) \oplus E_1^{(n)}(K) \oplus E_2^{(n)}(K) \to E(L)$ by
    $f= f_0 + f_1 + f_2.$
    Note that for $P \in E(L)$, $Q_0=P+P^\tau+ P^{\tau^2} \in E(L)^{\tau =1} = f_0(E(K))$, $Q_1=P+\zeta^2 P^\tau+ \zeta P^{\tau^2} \in E(L)^{\tau =\zeta} = f_1(E_1^{(n)}(K))$ and $Q_2=P+ \zeta P^\tau+ \zeta^2 P^{\tau^2} \in E(L)^{\tau =\zeta^2} = f_2(E_2^{(n)}(K))$ and hence $Q_0+Q_1+Q_2=3P  \in \mathrm{Image}(f)$.
    Consequently, for any prime $p \neq 3$, the map
    $f: E_0(K)[p^\infty] \oplus E_1^{(n)}(K)[p^\infty] \oplus E_2^{(n)}(K)[p^\infty] \to E(L)[p^\infty]$
    is surjective. Since $n \neq 1,2,4$ is a cube-free integer, it is well known that for any prime $p \neq 3$, $E_1^{(n)}(\Q)_{\tors}[p^\infty] \simeq E_2^{(n)}(\Q)_{\tors}[p^\infty] = 0$, and hence by \cite[Cor. 4 and Thm. 5(ii)]{gjt} $E_1^{(n)}(K)_{\tors}[p^\infty] \simeq E_2^{(n)}(K)_{\tors}[p^\infty] = 0$, we obtain $E(L)[p^\infty] = E(K)[p^\infty]$.\\
To conclude the result, we will show that $E(L)[3^\infty]= E(K)[3^\infty] = \{ \mathcal{O}, (0, \pm 1) \} \simeq \Z/3\Z$.  The $3$-division polynomial of $E$ is $\psi_3(x) = 3x(x^3+4)$ \cite [Chap. III, Ex. 3.7]{silverman} and hence every $Q' = (x_{Q'}, y_{Q'})\in E(L)[3]$ satisfies $\psi_3(x_{Q'})=0$. Consequently, $E(L)[3] = \{ \mathcal{O}, (0, \pm 1) \} = E(K)[3]$. Now suppose that $Q\in E(L)$ has order equal to $3^t$ for some $t\geq2$. Then $Q^{''}:=3^{t-2}Q'=(x'',y'')$ satisfies that $3Q^{''}\in\{(0,\pm1)\}$. As $3Q^{''} = \left( \frac{\phi_3(Q^{''})}{\psi_3^2(Q^{''})}, \frac{\omega_3(Q^{''})}{\psi_3^3(Q^{''})}\right)$ \cite[Chap. III, Ex. $3.7$]{silverman}, we see that $x''$ is a root of $\phi_3(x)= x^9 - 96x^6+48x^3+64$ which is an irreducible polynomial of degree $9$ over $K$, hence $x'' \notin L$, and consequently, $E(L)[3^\infty] = E(L)[3]$.\\
The explicit structure for $E(K)_{\tors}$ is well known. \qedhere
\end{proof}
\subsection{Preliminaries on  modular curve and modular automorphisms} We recall some basic facts about modular curves (cf. \cite{cube
sum}).
 Let $\mathcal{H}=\{z\in \C \mid  \mathrm{Im}(z) >0 \}$ and $U_0(36):=\left\{\begin{pmatrix}
    a&b\\c&d
\end{pmatrix} \in \gl_2(\widehat{\Z})\mid c\equiv 0 \mod 36 \right\}$ be the open subgroup of $\mathrm{GL_2(\widehat\Z)}$ and define  $\Gamma_0(36):=\mathrm{GL_2}^+(\Q)\\ \cap U_0(36)$. Recall  $X_0(36)$ is the modular curve over $\Q$ of level $\Gamma_0(36)$ and can be expressed as  the double coset space 
\[
\begin{aligned}
    X_0(36)(\C)&=\mathrm{GL_2}^+(\Q) \backslash (\mathcal{H} \sqcup \mathbb{P}^1(\Q)) \times \mathrm{GL_2(\mathbb{A}}_f) / U_0(36)\\
&\simeq \Gamma_0(36) \backslash \mathcal{H} \bigsqcup \Gamma_0(36)\backslash\mathbb{P}^1(\Q),
\end{aligned}
 \]
where $\mathbb{A}_f$ denotes the finite adele ring of $\Q$. For a subgroup $H$ of a group $G$, define $N_G(H)$ to be the normalizer of $H$ in $G$. %\Gamma_0(36)$ in $\mathrm{GL_2}^+(\Q)$. 
Given an algebraic curve $X/\Q$ and a field extension $F$ over $\Q$, we denote by $\mathrm{Aut}_F(X)$ the group of algebraic automorphisms of $X$ defined over $F$. The action of $N_{\mathrm{GL_2}^+(\Q)}(\Gamma_0(36))$ on $X_0(36)$, induced from fractional linear transformation, gives us an injective homomorphism \cite{cube sum} 
\begin{equation}\label{Psi}
\Psi: N_{\mathrm{GL_2}^+(\Q)}(\Gamma_0(36))/ \Q^\times\Gamma_0(36) \hookrightarrow  \mathrm{Aut}_{\C}(X_0(36)) \simeq \OO_K^\times \rtimes X_0(36)(\C).
\end{equation}
By \cite{akbas-singerman}, \cite{atkin-lehner}, the quotient group $N_{\mathrm{GL_2}^+(\Q)}/ \Q^\times\Gamma_0(36) \simeq S_3 \times A_4$ is a solvable group of order $72$, generated by the Atkin-Lehner involutions $w_4:=\begin{pmatrix}
    4 &-1\\
    36 &-8
\end{pmatrix}$ and $w_9:=\begin{pmatrix}
    9 & 2\\36 &9
\end{pmatrix}$, along with the {\it exotic automorphism}   $e:=\begin{pmatrix}
    1 &0\\
    6&1
\end{pmatrix}$ (see \cite{survey article}).\\
The modular curve $X_0(36)$ has genus $1$  and  
the elliptic curve $(X_0(36), [\infty])$ is isomorphic to $E:y^2=x^3+1$ over $\Q$. Furthermore, by \cite{elkies}, an explicit isomorphism can be  described using the Dedekind $\eta$-functions:
\begin{equation}\label{modular parametrization}
\begin{aligned}
    f': X_0(36) &\xrightarrow{\sim} E: y^2=x^3+1,\\
    \lambda &\mapsto (x(\lambda), y(\lambda))=\left(\frac{\eta(12\lambda)\eta^3(18\lambda)}{\eta(6\lambda)\eta^3(36\lambda)}, \frac{\eta^2(18\lambda)\eta^4(12\lambda)}{\eta^2(6\lambda)\eta^4(36\lambda)}\right).
\end{aligned}
\end{equation}
The modular curve $X_0(36)$ has a set $\mathcal{C}$ of $12$ inequivalent cusps and we can identify $f'(\mathcal{C})$ with $E(K)$ \cite{cube sum}.
Using SageMath, we numerically approximate the images of the cusps under $f'$ and have the following table:
\begin{table}[ht!]
\centering
\resizebox{\textwidth}{!}{
\begin{tabular}{|c|c|c|c|c|c|}
  \hline
   $f'(\infty) = \mathcal{O}$  & $f'(0) = (2,3)$ & $f'(\frac{1}{2}) = (2, -3)$   &
   $f'(\frac{1}{4}) = (-1,0)$  & $f'(\frac{1}{9}) = (0, -1)$ & $f'(\frac{1}{18}) = (0,1)$\\
     \hline
   $f'(\frac{1}{3}) = (2 \zeta, 3)$ & $f'(\frac{2}{3}) = (2 \zeta^2, 3)$ & $f'(\frac{1}{6}) = (2 \zeta^2, -3)$ &
   $f'(\frac{5}{6}) = (2 \zeta, -3)$ & $f'(\frac{1}{12}) = (-\zeta, 0)$ & $f'(\frac{5}{12}) = (-\zeta^2, 0)$\\
     \hline
\end{tabular}}.
\caption{Images of the cusps under $f'$}
\label{cusp}
\end{table}

\noindent Let $U=\left\{ \begin{pmatrix}
    a&b\\c&d
\end{pmatrix} \in \gl_2(\widehat{\Z})\mid\, a\equiv d \mod3,\,c\equiv 0 \mod 36\right\}$ be the index $2$ subgroup of $U_0(36)$. The corresponding  modular curve $X_U$  over $\Q$ can be expressed as  the double coset space  $$X_U(\C):= \mathrm{GL_2}^+(\Q) \backslash (\mathcal{H} \sqcup \mathbb{P}^1(\Q)) \times \mathrm{GL_2(\A}_f) / U.$$ Using class field theory, we can  identify
$\Q_+\widehat{\Z}^\times/\Q_+^\times\det(U)$ with $\Gal(K/\Q)$ (cf. \cite{cube sum} , \cite{hu-shu-yin}). As $\gl_2^+(\Q) \cap U=\Gamma_0(36)$, the modular curve $X_U$ is isomorphic to $X_0(36)\times_\Q K$ over $\Q$ \cite[Chap. $6$]{shimura}. Set $$\epsilon:=\begin{pmatrix}
    1&0\\0&-1
\end{pmatrix},\; \text{then}\;U_0(36)/U=\la\epsilon\ra .$$ The Galois group $\Gal(K/\Q)$ has nontrivial action on $X_U$ through the right translation of $\epsilon$ on $X_U$. We have $\mathrm{Aut_\Q}(X_U)=\mathrm{Aut}_K(X_0(36))\rtimes\Gal(K/\Q)\simeq \left(X_0(36)(K)\rtimes \OO_K^\times\right)\rtimes \Gal(K/\Q)$.  
There is a canonical homomorphism $N_{\gl_2(\A_f)}(U)/U\longrightarrow \mathrm{Aut}_\Q(X_U)$ induced from the right translation operations on $X_U$, i.e., for $P=[z,g]_U\in X_U$ and $x\in N_{\gl_2(\A_f)}(U)$, $P\mapsto P^x=[z,gx]_U.$ The curve $X_U$ has two connected components over $\C$ and an element of $N_{\gl_2(\A_f)}(U)$ maps one connected component onto the other if and only if it yields the value $-1$ as its image under the sequential composition of the following morphisms \cite{cube sum}: $$\gl_2(\A_f)=\gl_2^+(\Q)U_0(36)\xrightarrow{\det}\Q_+^\times\widehat{\Z}^\times \twoheadrightarrow \widehat{\Z}^\times\xrightarrow{\text{mod}\, 3}\Z_3^\times/(1+3\Z_3).$$ 
\subsection{The modular actions on Heegner points}\label{subsection 2.2} Let $p\equiv 4,7 \mod 9$ be an integer prime. Let $\rho:K \hookrightarrow M_2(\Q)$ be the normalized embedding with a fixed point $h=\frac{p\zeta}{6}\in \mathcal{H}$, i.e, we have
$\rho(t)\begin{pmatrix}
    h\\1
\end{pmatrix}=t\begin{pmatrix}
    h\\1
\end{pmatrix}$ for any $t\in K.$
 We can write $h=A\zeta$, where $A=\begin{pmatrix}
    p &0\\0&6
\end{pmatrix}$. Then the embedding $\rho$ is explicitly given by\[\rho(\zeta)= A\begin{pmatrix}
    -1&-1\\1&0
\end{pmatrix}A^{-1}=\begin{pmatrix}
    -1&-\frac{p}{6}\\
    \frac{6}{p} & 0
\end{pmatrix}.\]
For an integer $c\geq 1$, let $\OO_c$ denote the order of $K$ of conductor $c$ and let $H_c$ be the ring class field of conductor $c$. Then $\widehat{K}^\times\cap U=\widehat{\OO_{6p}}^\times$. Let $w=\begin{pmatrix}
    1&-p/6\\
    0 & -1
\end{pmatrix}\in N_{\gl_2(\Q)}(K^\times)$, which normalizes $U$ and hence $w$ represents an element  in $ \mathrm{Aut}_\Q(X_U)$. Similarly, $w\epsilon$ normalizes $U$ and thus gives us an element of $\mathrm{Aut}_{\Q}(X_U)$. Let $\OO_{K,2}$ and $\OO_{K,3}$ be the completions of $\OO_K$ at the unique place above $2$ and $3$, respectively. We have $$
K_2^\times/\Q_2^\times(1+2\OO_{K,2})=\la\zeta_2\ra^{\Z/3\Z}\quad \text{and}\quad \OO_{K,3}^\times/\Z_3^\times(1+3\OO_{K,3})=\la\zeta_3\ra^{\Z/3\Z},
$$
where $\zeta_2$ and $\zeta_3$ are the images of $\zeta$ in $\OO_{K,2}$ and $\OO_{K,3}$, respectively. Considering the natural embeddings of $\zeta_2$ (respectively $\zeta_3$) in $\gl_2(\A_f)$, we note that  $\zeta_2$ (respectively $\zeta_3$) normalize $U$ in $\gl_2(\A_f)$ and hence we have the embeddings $$
K_2^\times/\Q_2^\times(1+2\OO_{K,2}) \hookrightarrow \mathrm{Aut}_\Q(X_U)\quad \text{and}\quad \OO_{K,3}^\times/\Z_3^\times(1+3\OO_{K,3}) \hookrightarrow \mathrm{Aut}_\Q(X_U).
$$

\begin{proposition}\label{actions on P} 
  For any point $P\in X_U$, we have $$
 P^{\zeta_2^{-1} }=P+(0,1), \quad  P^{\zeta_3^{-1}}=\zeta P+(-\zeta,0),                 \quad P^{\zeta_3}=\zeta^2P+(-1,0),  \quad P^{w\epsilon}=-\zeta^2 P.
    $$
\end{proposition}
\begin{proof}
    Since $\zeta_2^{-1},\, \zeta_3^{-1},\, \zeta_3, \, w\epsilon$ all have determinant $\equiv 1\mod 3$, when identified as elements in $\mathrm{Aut}_\Q(X_U)$, they lie in the subgroup $\mathrm{Aut}_K(X_0(36))$. Let $P=[z,1]_U$ be a point on $X_U$. Note that $$
    \begin{pmatrix}
        1 &1/2\\
        18 & 10
    \end{pmatrix}\zeta_2^{-1}\in U, \quad \begin{pmatrix}
    1& -1/3\\
    12&-3
\end{pmatrix}\zeta_3^{-1}\in U, \quad \begin{pmatrix}
    3&-1/3\\
    12& -1
\end{pmatrix}\zeta_3\in U, \quad \begin{pmatrix}
    1 & p/6\\
    0 & 1
\end{pmatrix}w\epsilon \in U.
    $$
It is easy to see  the embedding in (\ref{Psi}) in fact gives the embedding $\Psi: N_{\mathrm{GL_2}^+(\Q)}(\Gamma_0(36))/ \Q^\times\Gamma_0(36) \\ \hookrightarrow \OO_K^\times \rtimes \mathcal{C}$. Then $P^{\zeta_2^{-1}}=[z,\zeta_2^{-1}]_U=\left[ \begin{pmatrix}
        1 &1/2\\
        18 & 10
    \end{pmatrix}z,1\right]_U=\Psi \begin{pmatrix}
        1 &1/2\\
        18 & 10
    \end{pmatrix}P.$ \\
Taking $P=[\infty]$ and $P=[0]$ in turn, it follows from Table \ref{cusp} that $\Psi \begin{pmatrix}
        1 &1/2\\
        18 & 10
    \end{pmatrix}P=P+(0,1)$. The other identities can be obtained in a similar manner. \qedhere
\end{proof}

Let $\sigma: \widehat{K}^\times \rightarrow \Gal(K^{\mathrm{ab}}/K)$ be the Artin reciprocity map and put  $\sigma_t: =\sigma(t)$ for $t\in \widehat{K}^\times$. Recall $h =\frac{p\zeta}{6}$ and let $P_0=[h,1]$ be the CM point on $X_0(36)$. Also recall $f'$ denotes the identity morphism from $X_0(36)$ to $E$ as in (\ref{modular parametrization}) and we denote $P_1:=f'(P_0)$.

\begin{corollary}{\label{action on P_0}}
    The point $P_1\in E(H_{6p})$ satisfies $$
 P_1^{\sigma_{\zeta_2}^{-1} }=P_1+(0,1), \quad  P_1^{\sigma_{\zeta_3}^{-1}}=\zeta P_1+(-\zeta,0), \quad P_1^{\sigma_{\zeta_3}}=\zeta^2P_1+(-1,0),  \quad P_1^{\sigma_{w\epsilon}}=-\zeta^2 P_1=\overline{P_1}.$$
\end{corollary}
\begin{proof}
    By Shimura's reciprocity law \cite[Thm. $6.31$ and Thm. $6.38$]{shimura}, $P_0^{\sigma_t}=[h, \rho(t)]$ for any $t\in \widehat{K}^\times$.
Since $\widehat{K}^\times\cap \,U=\widehat{\OO_{6p}}^\times$, $P_0$ is defined over the ring class field $H_{6p}$. The Galois actions follow from Proposition \ref{actions on P}. \qedhere
\end{proof}
From now on, we take $L:=K(\sqrt[3]{p})$. 
\begin{proposition}{\label{field}}
We have $H_{6p}=H_{3p}(\sqrt[3]{2})$ with $\Gal(H_{6p}/H_{3p})=\la\sigma_{\zeta_2}^{-1}\ra\simeq \Z/3\Z$ and $\left(\sqrt[3]{2}\right)^{\sigma_{\zeta_2}^{-1}-1}=\zeta.$ Also $L \subset H_{3p}$, $[H_{3p}:L]=\frac{p-1}{3}$ and $\Gal(L/K)=\la{\sigma_{\zeta_3}}\ra \simeq \Z/3\Z$ with $\left(\sqrt[3]{2}\right)^{\sigma_{\zeta_3}-1}=\zeta$. Moreover,

\begin{enumerate}[label=\arabic*.]
 
\item For $p\equiv4\mod 9$, we have $\left(\sqrt[3]{p}\right)^{\sigma_{\zeta_3}^{-1}-1}=\zeta$.
\item For $p\equiv7\mod 9$, we have $\left(\sqrt[3]{p}\right)^{\sigma_{\zeta_3}-1}=\zeta$.
\end{enumerate}
\end{proposition}
\begin{proof}
    The Galois group, 
    $\Gal(H_{6p}/H_{3p}) \simeq K^\times\widehat{\OO_{3p}}^\times/K^\times\widehat{\OO_{6p}}^\times$ is cyclic of order $3$ and generated by $\sigma_{\zeta_2}$ and equivalently by $\sigma_{\zeta_2}^{-1}$. We have $$
\left(\sqrt[3]{2}\right)^{\sigma_{\zeta_2}-1}=\left(\frac{\zeta_2, 2}{K_2; 3} \right)=\zeta^{-\frac{4-1}{3}}=\zeta^2,
$$
where $\left(\frac{., .}{K_w;\, 3} \right)$ denotes the third Hilbert symbol over $K_w$ \cite[Chap. V, \S$3$]{neukrich}. The stated relation between $H_{3p}$ and $L$  follows from class field theory. Next, let $v$ and $\overline{v}$ be the two places of $K$ above $p$. Then, $$
\left(\sqrt[3]{p}\right)^{\sigma_{\zeta_3}-1}=\left(\frac{\zeta_3, p}{K_3; 3} \right)=\left(\frac{\zeta_v, p}{K_v; 3} \right)^{-1}. \left(\frac{\zeta_{\overline{v}}, p}{K_{\overline{v}}; 3} \right)^{-1}=\zeta^{-\frac{p-1}{3}}=\begin{cases}
    \zeta^2, \quad p\equiv 4\mod{9},\\
    \zeta, \,\,\quad p\equiv 7\mod{9}.
\end{cases}
$$
Similarly, $\left(\sqrt[3]{2}\right)^{\sigma_{\zeta_3}-1}=\left(\frac{\zeta_3,\, 2}{K_3;\, 3} \right)=\left(\frac{\zeta_2,\, 2}{K_2;\, 3} \right)^{-1}=\zeta.$ In other words, $\Gal(L/K)$ is generated by $\sigma_{\zeta_3}$. \qedhere
\end{proof}
In summary, we have the following diagram of field extensions. 
\begin{figure}[H]  
\[
\begin{tikzpicture}[scale=1.0, transform shape, line width=0.9pt]

\node (A) at (0,0) {$H_{6p}=H_{3p}(\sqrt[3]{2})$};
\node (B) at (0,-1.5) {$H_{3p}$};
\node (C) at (0,-3.5) {$L=K(\sqrt[3]{p})$};
\node (D) at (0,-5) {$K=\Q(\zeta)$};
\node (E) at (0,-6.5) {$\Q$};
\node (F) at (3,-2.7) {$K(\sqrt[3]{2})=H_6$};

\draw[-] (A) -- (B) node[midway, left] {$3$};
\draw[-] (A) -- (F)node[midway, right] {$\;\;p-1$};
\draw[-] (B) -- (C)node[midway, left] {$\frac{p-1}{3}$};
\draw[-] (C) -- (D)node[midway, left] {$3$};
\draw[-] (D) -- (E)node[midway, left] {$2$};
\draw[-] (D) -- (F)node[midway, right] {$\;\;3$};

\end{tikzpicture}
\]
\end{figure}
Define  $T := (- \zeta\sqrt[3]{4}, - \sqrt{-3}) \in E(H_6)$.
By Proposition \ref{field}, $T^{\sigma_{\zeta_2}^{-1}}=T+(0,1)$. 
Then letting \[Q:=P_1-T,\] we find by Corollary \ref{action on P_0} that $Q^{\sigma_{\zeta_2}^{-1}}=Q $. So $Q\in E(H_{3p})$. Put $R_1:=\tr_{H_{3p}/L} Q\in E(L)$.\\
Recall, $q$ is defined in \eqref{qdef} and let  $k \in \{ 1, 2 \}$ be the integer such that $q = p^k$.
\begin{proposition}\label{action on R_1}Define $\delta_p \in \Gal(L/K)$ by $\delta_p= \begin{cases}
       \sigma_{\zeta_3}^{-1} & \text{ if } p \equiv 4 \mod 9, \\
       \sigma_{\zeta_3} & \text{ if } p \equiv 7 \mod 9.
     \end{cases}$\\
Then $(\sqrt[3]{p})^{\delta_p} = \zeta \sqrt[3]{p}$ and $R_1^{\delta_p} = \zeta^k R_1$.
 \end{proposition}
 \begin{proof}
    This is a consequence of Corollary \ref{action on P_0}, Proposition \ref{field} and the fact that $\tr_{H_{3p}/L}(-\zeta^i,0)=\frac{p-1}{3}(-\zeta^i,0)$ for $i=0,1$.  \qedhere
\end{proof}
\section{Nontriviality of the Heegner point}\label{sec:nontorsion}
We prove the point $R_1\in E(L)$ is nontorsion by considering its reduction modulo $p$. 
\begin{lemma}\label{P'}
    Let $u = 1 + \sqrt[3]{2} + \sqrt[3]{4} $ be the fundamental unit of $\Q(\sqrt[3]{2})$. Then we have $f'\left(\frac{\zeta}{6}\right) = \left(-u\zeta^2 \sqrt[3]{4}, u(1+ \sqrt[3]{2}) \sqrt{-3}\right) :=P' \in E(H_6)$.
\end{lemma}
\begin{proof}
     As $\frac{\zeta}{6} \in X_0(36)(\C)$ corresponds to the isogeny $\la\frac{\zeta}{6}\ra \to \la6\zeta\ra$, we get that $f'\left(\frac{\zeta}{6}\right) \in E(H_6)$. We know that $E(H_6)$ has rank $0$ and in fact $E(H_6) \simeq \frac{\Z}{6\Z}\oplus \frac{\Z}{6\Z}$, so $f'\left(\frac{\zeta}{6}\right)$ must be one of the $36$ torsion points of $E(H_{6})$. Numerically, using SageMath, we see $f'\left(\frac{\zeta}{6}\right)$ is approximately $P'$ and hence $f'\left(\frac{\zeta}{6}\right) = P'$. \qedhere 
\end{proof}

\begin{theorem}\label{nontorsion}
 Assume that $2$ is not a cube in $\F_p$. Then
  $R_1$ is a nontorsion point in $E(L)$. In particular, the point $S_1=\tr_{H_{6p}/L}P_1=3R_1\in E(L)$ is nontorsion.

\end{theorem}

\begin{proof}
Recall from Proposition \ref{action on R_1}, $R_1 \in E(L)$ satisfies $R_1^{\delta_p} = \zeta^kR_1$.
Further note that the points $S$ in $E(L)_{\tors}$ which satisfies $S = S^{\delta_{p}} = \zeta^k S $ is given by $ E(L)[3] =\{\OO, (0,\pm1)\}$. It suffices to show $R_1 \notin E(L)[3]$ to conclude $R_1$ is a nontorsion point in $E(L)$.\\
Since $x(\lambda), y(\lambda)$ in (\ref{modular parametrization}) are modular functions as well as eta quotients, by Ligozat's criterion \cite{ligozat}, they have $q$-expansion with integral coefficients.
Next consider $\frac{\zeta}{6} \in X_0(36p)$, which corresponds to the isogeny $\la\frac{\zeta}{6}\ra \to \la 6p\zeta \ra $. The elliptic curve corresponding to $\frac{\zeta}{6}$ has conductor $6$ and the elliptic curve corresponding to $6p \zeta$ has conductor $6p$.
Then by \cite[Prop. $5.2.1$]{mock heegner},
$$x\left(\frac{\zeta}{6}\right) \equiv x\left(\frac{p\zeta}{6}\right)^p \mod{p\OO_{H_{6p}}} \text{ and } y\left(\frac{\zeta}{6}\right) \equiv y\left(\frac{p \zeta}{6}\right)^p \mod{p\OO_{H_{6p}}}.$$
The prime $p$ splits in $\OO_K$ as $p\OO_K=\p\overline{\p}$ and both the primes $\p$ and $\overline{\p}$ are totally ramified in $H_{3p}/K$. Since $2$ is not a cube in $\F_p$, the primes $\mathfrak{P}$ and $\overline{\mathfrak{P}} $, above $\p$ and $\overline{\p}$, respectively, remain inert in $\OO_{H_{6p}}$. Note that $\OO_{H_{3p}}/ \mathfrak{P} \simeq \OO_{H_{3p}}/ \mathfrak{\overline{P}} \simeq \F_p$ and $\OO_{H_{6p}}/ \mathfrak{P} \simeq \OO_{H_{6p}}/ \mathfrak{\overline{P}} \simeq \F_p(\sqrt[3]{2})$. Also $E$ has good reduction at $\mathfrak{P}$ and $\mathfrak{\overline{P}}$. We continue to   denote by $E$, the reduction  of $E$ modulo $\mathfrak P$ and $\mathfrak{\overline{P}}$. Then,
$$P' = \left(x\left(\frac{\zeta}{6}\right),y\left(\frac{\zeta}{6}\right)\right) \equiv \left(x\left(\frac{p \zeta}{6}\right)^p, y\left(\frac{p \zeta}{6}\right)^p\right) = \fr(P_1) \in E\left(\F_p(\sqrt[3]{2})\right) \times E\left(\F_p(\sqrt[3]{2})\right)$$
and hence
$$P' - \fr(T)=   \fr(P_1)-\fr(T)=\fr(Q) \in E\left(\F_p(\sqrt[3]{2})\right) \times E\left(\F_p(\sqrt[3]{2})\right).$$
Recall from \ref{subsection 2.2} that $Q \in E(H_{3p})$ and hence
 $ Q = \fr(Q) \in E(\F_{p}) \times E(\F_p)$. Further, from Lemma \ref{P'},  $6P' = \OO \in E(H_{6p})$. Moreover,  $T \in E(H_6)$ and hence $6T=\OO$. Consequently, $(\OO, \OO) = 6 \fr({Q}) = 6 Q \in E(\F_p) \times E(\F_p)$.\\ 
Write $p = 1+ 2^r \cdot 3m$ with $(m,6)=1$ and $r \ge 1$. As the primes above $p$ are totally ramified in $H_{3p}/L$ \cite{mock heegner}, 
\begin{equation}
R_1 = Tr_{H_{3p}/L} Q \equiv 2^rm Q \equiv  \pm 2Q  \in E(\F_p) \times E(\F_p).
\end{equation}
Next,
$$\fr(T) = \left( - \zeta \sqrt[3]{4} (4^{\frac{p-1}{3}}), - \sqrt{-3} (\sqrt{-3})^{p-1}\right)\in E\left(\F_p(\sqrt[3]{2})\right) \times E\left(\F_p(\sqrt[3]{2})\right).  $$
Observe that $(\sqrt{-3})^{p-1} = 1 \in \F_p$ and $4^{\frac{p-1}{3}} = v \in \F_p$ such that $v^2+v+1 =0 \in \F_p$.
We fix an isomorphism $\alpha: \OO_K/p\OO_K \to \F_p \times \F_p$ given by $\zeta \mapsto (v,v^2)$. Under the isomorphism $\alpha$, we have that
$$\fr(T) = \left((- v^2 \sqrt[3]{4}, - (1+2v)), (-  \sqrt[3]{4}, - (1+2v^2))\right) \in E\left(\F_p(\sqrt[3]{2})\right) \times E\left(\F_p(\sqrt[3]{2})\right) \quad \  \text { and} $$
$$P'= \left((- v^2u \sqrt[3]{4}, u(1+ \sqrt[3]{2}) (1+2v)) , (-vu \sqrt[3]{4}, u(1+ \sqrt[3]{2}) (1+2v^2))\right) \in  E\left(\F_p(\sqrt[3]{2})\right) \times E\left(\F_p(\sqrt[3]{2})\right).$$
Thus we compute $Q = \fr(Q) = \left((-v^2,0), (2v, -3)\right) \in E(\F_p) \times E(\F_p)$ and hence 
\begin{equation}\label{r1modp}R_1 =\pm2\left( (-v^2, 0),  (2v, -3)\right)=\left(\OO,(0,\pm1)\right) \in E(\F_p) \times E(\F_p).\end{equation}
It is plain that the images of $\OO $ and $(0,\pm1)$ are $(\OO, \OO) $ and $\big((0,\pm1), (0,\pm1)\big)$  in $E(\F_p) \times E(\F_p)$, respectively and consequently, $R_1\notin E(L)[3]$.\\
Lastly, $S_1=\mathrm{Tr}_{H_{6p}/L}(P_1- T+ T)=3R_1+\tr_{H_{6p}/L}( T)$. However, $\tr_{H_{6p}/H_{3p}}( T)=(1+\zeta+\zeta^2)T=0$ and thus  $S_1$ is nontorsion. \qedhere \end{proof}

Recall   $k \in \{ 1, 2 \}$ be the integer such that $q = p^k$. Also recall that by Proposition \ref{action on R_1}, $R_1 \in E(L)$ satisfies $R_1^{\delta_p} = \zeta^kR_1$ and by the general theory of twisting, we have an isomorphism $\phi_k: E(L)^{\delta_p = \zeta^k} \simeq E_1^{(q)}(K)$.
Define $R:=\phi_k(R_1) \in E_1^{(q)}(K)$. Then $R$ is a nontorsion point in $E_1^{(q)}(K)$. Set $Z_1 := (1-\zeta^2)R_1 \in E(L)^{\delta_p = \zeta^k}$ and put $Z:=\phi_k(Z_1) = (1-\zeta^2)R \in E_1^{(q)}(K)$. Then $Z$ is also a nontorsion point in $E_1^{(q)}(K)$.

\begin{lemma}\label{phi(Z_1)}
With the notation as above, $Z \in E_1^{(q)}(\Q)$ is a nontorsion point. In particular, $2q$ is a sum of two rational cubes.
\end{lemma}

\begin{proof} Recall from Corollary \ref{action on P_0}, we have $\overline{P}_1 = -\zeta^2P_1$. Note that $\overline{T} = - \zeta T$ and hence $\overline{Q}  = \overline{P_1-T}=-\zeta^2 Q + (0, 1) $, which in turn implies that $\overline{R_1} = -\zeta^2R_1 + (0 ,\pm 1)$. As a consequence, we see that $R_1 + \overline{R_1} = Z_1 + (0, \pm 1)$. As $Z_1 \in E(L)^{\delta_p = \zeta^k}$, we get that $R_1 + \overline{R_1} \in E(L)^{\delta_p = \zeta^k}$ and clearly  $R_1+ \overline{R_1}$ is fixed by   $\Gal(K/\Q)$. Now, $  R + \overline{R} = \phi_k(R_1+ \overline{R_1} )  = \phi_k( Z_1 + (0, \pm 1) ) = Z + (0, \pm q) \in E_1^{(q)}(\Q) $. It follows that $Z\in E_1^{(q)}(\Q)$  is a nontorsion point. \qedhere    
\end{proof}

\begin{lemma}\label{no R_1=aX+Y}
 There does not exist $X \in E(L)^{\delta_p = \zeta^k}$ and $Y \in E(L)^{\delta_p = \zeta^k}_{\tors}$ such that
  $$R_1 = (1-\zeta)X + Y.$$  
\end{lemma}
\begin{proof}
Recall that $E(L)^{\delta_p = \zeta^k}_{\tors}= \{ \OO, (0, \pm 1)\}$. If possible, assume that $R_1 = (1-\zeta)X + Y, \text{ with } R_1,X \\\in E(L)^{\delta_p = \zeta^k}\; \text{and}\; Y \in E(L)^{\delta_p = \zeta^k}_{\tors}.$
As the primes $\p,\,\overline{\p}$ over $p$ are totally ramified in $H_{3p}/K$, it follows that
$X^{\delta_p} \equiv X \in E(\F_p) \times E(\F_p)$.
Since $X \in E(L)^{\delta_p = \zeta^k}$, we have $X^{\delta_p} = \zeta^k X$ and hence
$ (1-\zeta^k)X \equiv \OO \in E(\F_p) \times E(\F_p),$ 
which in turn implies 
$   (1-\zeta)X  \equiv \OO \in E(\F_p) \times E(\F_p).$ So we deduce 
$$ R_1 \equiv Y \in E(\F_p) \times E(\F_p).$$
Then from  \eqref{r1modp}, we deduce that
 $Y \equiv (\OO, (0, \pm 1)) \in  E(\F_p) \times E(\F_p),$
which is a contradiction to the fact that $ Y \in E(L)^{\delta = \zeta^k}_{\tors}= \{ \OO, (0, \pm 1)\}$. \qedhere
\end{proof}
\section{Explicit Gross-Zagier Formulae  and the BSD Invariants  }\label{section 3}
We briefly recall the setup and notation required for discussing the Gross-Zagier formula, following \cite{G-Z formula}. Recall, $\mathbb{A}_\Q$ and $\mathbb{A}_K$ denote the ring of adeles over $\Q$ and $K$, respectively. Let $\pi$ denote the automorphic representation of $\gl_2(\mathbb{A}_\Q)$ associated to $E/\Q$. Let $\chi:\Gal(H_{6p}/K)\rightarrow \C^\times$ be the cubic character given by $\chi(\tau)=(\sqrt[3]{p})^{\tau-1}$. Set 
\[
L(s, E, \chi):=L(s-1/2, \pi_K \otimes\chi), \hspace{1cm}\] 
where $\pi_K$ denotes the base change of $\pi$ to $\gl_2(\mathbb{A}_K)$, and $L(s-1/2, \pi_K \otimes\chi)$ is the automorphic $L$-function attached to $\pi_K \otimes\chi$, respectively.  The $L$-function $L(s, E, \chi)$ satisfies \[L(s, E, \chi)=L(s, E_1^{(q)}) L(s,E_2^{(q)}),\]
where,  for $k=1,2$, $L(s, E_k^{(q)})$ is the complex $L$-function  of $E_k^{(q)}$.\\
The elliptic curve $E$ has conductor $36$.  Now recall $h=\frac{p\zeta}{6} \in \mathcal{H}$, $P_0=[h,1]_{U_0(36)}$ is the CM point on $X_0(36)(H_{6p})$ and $f'(P_0)=P_1\in E(H_{6p})$. Define the Heegner point \[
S_1:=\tr_{H_{6p}/L}P_1 \in E(L).
\]
Let $\Omega_k^{(q)}$ denote the minimal real period of $E_k^{(q)}$ defined by $\Omega_k^{(q)}=\int_{E_k^{(q)}(\R)}|\omega_{E_k^{(q)}}|$, where $\omega_{E_k^{(q)}}$ is the invariant differential on the minimal model of $E_k^{(q)}$. Let $\widehat{h}_{\Q}(.)$ be the N\'eron-Tate height on $E_1^{(q)}(\Q)$.

\begin{theorem}\label{G-Z}
    Let $p\equiv4,7 \mod 9$ be a prime number such that $2$ is not a cube modulo $p$. We have the following explicit height formula of the Heegner point:
    \[
    \frac{L'(1, E_1^{(q)})L(1,E_2^{(q)})}{\Omega_1^{(q)}\Omega_2^{(q)}}=\frac{1}{3} \widehat{h}_\Q(S_1).
    \]
     
\end{theorem}

\begin{proof}
   We would apply variation of Gross-Zagier formula \cite[Thm. $1.6$]{G-Z formula} to compute the $L'(1, E, \chi)$. Recall we have $P_0=\frac{p \zeta}{6}$ and  $f': X_0(36) \to E$. Whereas \cite[Thm. 1.6]{G-Z formula} holds more generally for an abelian variety over a totally real field parametrized by certain Shimura curve, the Example \cite[p. $2535$]{G-Z formula} following the stated theorem applies in our setting and we obtain the following: 

   \begin{equation}\label{L'}
    L'(1,E,\chi)=9\cdot\frac{(8\pi^2)\cdot(\phi,\phi)_{\Gamma_0(36)}}{\sqrt{3}p}\cdot\la P_{\chi}^0(f'), P_{\chi^{-1}}^0(f')\ra_{K,K}.
\end{equation}
   
  \[ \text{Here } \  \quad 
    P_\chi^0(f')=\frac{\#\mathrm{Pic}(\OO_p)}{\mathrm{Vol(K^\times\widehat{\Q}^\times}\backslash\widehat{K}^\times,dt)}\int_{K^\times\widehat{\Q}^\times\backslash \widehat{K}^\times} f'(P_0)^{\sigma_t}\chi(t)dt =\frac{1}{9}\sum\limits_{t\in \mathrm{Pic}(\OO_{6p})}f'(P_0)^{\sigma_t}\chi(t) , 
    \]
 and $P_{\chi^{-1}}^0(f') =\frac{1}{9}\sum\limits_{t\in \mathrm{Pic}(\OO_{6p})}f'(P_0)^{\sigma_t}\chi^{-1}(t)$  are the Heegner cycles. Also $\phi \in S_2(\Gamma_0(36))$ is the newform associated with $E$, and $(\phi, \phi)_{\Gamma_0(36)}$ is the Peterson norm of $\phi$. Here the bilinear pairing $\la\cdot,\cdot\ra_{K,K}$ corresponds to the $K$-linear map $E(\overline{K})_\Q \otimes_K E(\overline{K})_\Q \rightarrow \C$  such that $\tr_{\C/\R}\la\cdot,\cdot\ra_{K,K}$ is the $\Q$-linear N\'eron-Tate height pairing $E(\overline{K})_\Q \otimes_\Q E(\overline{K})_\Q \rightarrow \R$ over the base field $K$ \cite[p. 789]{cube sum}.\\
Let $\Omega$ be the minimal real period of $E$. Then by noting the transformations of minimal real period under cubic twisting, we obtain
\begin{equation}\label{period product}
    \Omega_1^{(q)}\Omega_2^{(q)}=\Omega^2/p.
\end{equation}
Using SageMath, we compute that $\{\Omega, \Omega \cdot \left(\frac{1}{2}+\frac{\sqrt{-3}}{6}\right)\}$ is a $\Z$-basis of the period lattice $L$ of the minimal model of $E$. So we have
\begin{equation}\label{newform phi}
    \frac{1}{\sqrt{3}}\Omega^2=2\int_{\C/L} dxdy=\int_{E(\C)}|\omega_E \wedge \overline{\omega}_E|=8\pi^2(\phi, \phi)_{\Gamma_0(36)}.
\end{equation}
Consequently, \eqref{L'} simplifies to 
\begin{equation}\label{L'new}
    L'(1,E,\chi)=3 \Omega_1^{(q)}\Omega_2^{(q)}  \la P_{\chi}^0(f'), P_{\chi^{-1}}^0(f')\ra_{K,K}.
\end{equation}
Further using  the definition of $S_1$ and the properties of the bilinear pairing (see \cite[ $(4.6)$]{hu-shu-yin} for more details), we deduce 
 \begin{equation}\label{N-T pairing}
 \begin{aligned}
    \la P_\chi^0(f'), P_{\chi^{-1}}^0(f') \ra_{K,K}&=\frac{1}{9^2} \left\langle\sum\limits_{\delta\in \Gal(L/K)} S_1^\delta \chi(\delta), \sum\limits_{\delta\in \Gal(L/K)} S_1^\delta \chi^{-1}(\delta) \right\rangle_{K,K}\\
   &=\frac{1}{27}\left(\la S_1, S_1\ra_{K,K}-\la S_1, S_1^{\delta_p}\ra_{K,K} \right),
\end{aligned}
\end{equation}
where $\delta_p$ is the generator of $\Gal(L/K)$ (see Prop. \ref{action on R_1}). Next, from the same proposition, we have $S_1^{\delta_p}=\zeta^kS_1$. As   $|1+\zeta^k|=|\zeta^k|=1,$ observe that $\widehat{h}_K((1+\zeta^k)S_1)=\widehat{h}_K(\zeta^kS_1)=\widehat{h}_K(S_1)$. Thus we compute 
$$\la S_1, S_1^{\delta_p}\ra_{K,K}=\frac{1}{2}\left( \widehat{h}_K((1+\zeta^k)S_1)-\widehat{h}_K(\zeta^kS_1)-\widehat{h}_K(S_1)\right) =-\frac{1}{2}\widehat{h}_K(S_1).$$ 
 Also $\la S_1, S_1\ra_{K,K} = \frac{1}{2}\left(4\widehat{h}_K(S_1)-2\widehat{h}_K(S_1)\right)=\widehat{h}_K(S_1)$. Therefore, 
 \begin{equation}\label{<.>}
     \la P_\chi^0(f'), P_{\chi^{-1}}^0(f') \ra_{K,K}=\frac{1}{18}\widehat{h}_K(S_1)=\frac{1}{9}\widehat{h}_\Q(S_1).
 \end{equation}
Putting all these together, from (\ref{L'new}) and (\ref{<.>}), we deduce the stated result.
\end{proof}

Recall that there is an isomorphism $\phi_k:E(L)^{\delta_p=\zeta^k} \rightarrow E_1^{(q)}(K)$. Put $\phi_k(S_1):=S$.
\begin{corollary}\label{h(S)}
Let $p \equiv 4,7 \mod9$ be a prime. Then we have
$\frac{L'(1, E_1^{(q)})L(1,E_2^{(q)})}{\Omega_1^{(q)}\Omega_2^{(q)}} =\frac{1}{3}\widehat{h}_\Q(S). \qedhere$

\end{corollary}

Now we are ready to complete the proofs of Theorems \ref{main} and \ref{mainthm2}.

\begin{proof}[Proof of Theorem \ref{mainthm2}]
As $2$ is not a cube modulo $p$,  $S_1$ is a nontorsion point (Thm. \ref{nontorsion}).  It follows from Theorem \ref{G-Z} that $L'(1, E_1^{(q)}) L(1, E_{2}^{(q)}) \neq 0$ and hence $L(1, E_{2q^2})= L(1, E_{2}^{(q)}) \neq 0$. Note that by Coates-Wiles and Rubin's \cite{rubin} result  $\Sh(E_{2q^2}/\Q)$ is  finite. Now the full BSD formula  for $E_{2q^2}$ over $\Q$ holds by the work of Burungale-Flach \cite{burungale-flach}. \qedhere
\end{proof}

\begin{proof}[Proof of Theorem \ref{main}]
As $2$ is not a cube modulo $p$, the point  $S_1$ is nontorsion. Thus we see that $L'(1, E_{1}^{(q)}) L(1, E_{2}^{(q)}) \neq 0$. This shows $L'(1, E_1^{(q)}) \neq 0$, and  by the results of Gross-Zagier \cite{gross-zagier} and Kolyvagin   \cite{kolyvagin}, and Rubin \cite{rubin}, the rank part of the BSD conjecture holds and $\Sh(E_{1}^{(q)}/\Q)$ is finite.\\ 
Using Tate's algorithm, we see that the Tamagawa numbers of $E_1^{(q)}$ and $E_2^{(q)}$ over $\Q$ are given as follows: $c_v(E_1^{(q)})=c_v(E_2^{(q)})=3$ for the places $v\mid 2p$ and for all other places $v$, $c_v(E_1^{(q)})=c_v(E_2^{(q)})=1$. Also note that $E_1^{(q)}(\Q)_\tors\simeq \Z/3\Z$ and $E_2^{(q)}(\Q)_\tors\simeq \Z/3\Z$.\\ 
Let $W$ be a generator of the free part of $E_1^{(q)}(\Q)$. The full BSD formula for the rank $1$ curve $E_1^{(q)}$ predicts that:
\begin{equation}\label{BSD for E_1^q}
\frac{L'(1,E_1^{(q)})}{\Omega_1^{(q)}\cdot \widehat{h}_\Q(W)}=   \frac{ |\Sh(E_1^{(q)}/\Q)| \cdot \prod_v c_v(E_1^{(q)})}{|E_1^{(q)}(\Q)_\tors|^2} = |\Sh(E_1^{(q)}/\Q)|.
\end{equation}
For a prime $\ell \nmid 6p$, the elliptic curve $E_1^{(q)}$ has good reduction at $\ell$ and by works of Perrin-Riou \cite{perrin-riou} and Kobayashi \cite{kobayashi} the $\ell$-part of the full BSD formula holds for $E_1^{(q)}$, that is $\ell$-adic valuation of both sides of \eqref{BSD for E_1^q} are same. Since $E_1^{(q)}$ has potentially good, ordinary reduction at $p$, the $p$-part of the full BSD formula for $E_1^{(q)}$ is also true by the work of Li, Liu, and Tian \cite{li-liu-tian}. We want to show that the $3$-adic valuation of both sides of \eqref{BSD for E_1^q} are the same. 

The full BSD formula for the rank $0$ curve $E_2^{(q)}$ holds (by Thm. \ref{mainthm2}) and is given by
\begin{equation}\label{BSD for E_2^q}
 \frac{L(1,E_2^{(q)})}{\Omega_2^{(q)}} = \frac{ |\Sh(E_2^{(q)}/\Q)| \cdot \prod_v c_v(E_2^{(q)})}{|E_2^{(q)}(\Q)_\tors|^2}  = |\Sh(E_2^{(q)}/\Q)|,
\end{equation}
and hence the $3$-adic valuation  of both sides of \eqref{BSD for E_2^q} are same. Thus showing that the $3$-adic valuation  of both sides of \eqref{BSD for E_1^q} are same is equivalent to showing that the $3$-adic valuation of both sides of the following identity
\begin{equation*}\label{BSD product}
\frac{L'(1, E_1^{(q)})L(1,E_2^{(q)})}{\Omega_1^{(q)}\Omega_2^{(q)}  \widehat{h}_\Q(W)} = |\Sh(E_1^{(q)}/\Q)|\cdot |\Sh(E_2^{(q)}/\Q) |,
\end{equation*}
 are the same and hence by Corollary \ref{h(S)}, it is equivalent to show that the $3$-adic valuation of both sides of the equation
\begin{equation}\label{S-P}
|\Sh(E_1^{(q)}/\Q)| \cdot|\Sh(E_2^{(q)}/\Q) |=\frac{1}{3} \frac{\widehat{h}_\Q(S)}{\widehat{h}_\Q(W)},
\end{equation}
 are the same. As $2$ is not a cube modulo $p$,  $\dim_{\F_3} \ \sel^3(E_1^{(q)}/\Q) =2$ and $\dim_{\F_3} \ \sel^3(E_2^{(q)}/\Q) =1$ \cite{cube sum jnt} and we have,
$$ |\Sh(E_1^{(q)}/\Q)[3^\infty]| = |\Sh(E_2^{(q)}/\Q)[3^\infty] |=1. 
$$
Recall from Section \ref{sec:nontorsion},  $Z_1 = (1-\zeta^2) R_1$, $\phi_k(Z_1)=Z$, and  $S_1=3R_1 = (1-\zeta)Z_1$. It follows that $ \widehat{h}_\Q(S) = 3 \widehat{h}_\Q( Z)$. Therefore, to prove the $3$-adic valuations of both sides of (\ref{S-P}) are the same, it suffices to show that
$\widehat{h}_\Q(Z) = u \widehat{h}_\Q(W)$, where $u\in \Z_3^\times \cap \Q$.\\
By Lemma \ref{phi(Z_1)}, $Z$ is a nontorsion point of $E_1^{(q)}$. Thus $Z=aW + b (0, q) \in E_1^{(q)}(\Q),\; a,b \in \Z$ and hence $\widehat{h}_\Q(Z) = a^2 \widehat{h}_\Q(W)$. Consequently, it reduces to show $3 \nmid a$. 
If possible, let $a=3c$ with $c\in \Z$, then writing $cW := W_1$, we get
$ Z = 3W_1+ b(0,q)$. Thus $$3R= (1-\zeta)Z = 3(1 -\zeta)W_1 + b (1-\zeta)(0,q) = 3(1-\zeta)W_1,$$ which implies $R=(1-\zeta)W_1+d(0,q)$ for $d\in \Z$, i.e., we obtain $R_1=(1-\zeta)X+Y$ with $X \in E(L)^{\delta_p= \zeta^k} \text{ and } Y \in E(L)^{\delta_p = \zeta^k}_{\tors}$, which is a contradiction by Lemma \ref{no R_1=aX+Y}. This concludes the proof that for primes $\ell \neq 2$ the $\ell$-part of BSD formula holds for $E_1^{(q)}$.

As $E_1^{(q)}$ and $E_{2q}$ are $\Q$-isogenous and the BSD formula is isogeny invariant \cite[Chap. I, \S7]{milne}, the same result holds for $E_{2q}$ as well.\qedhere
\end{proof}


\begin{thebibliography}{abcde}
    \bibitem[ABS-BS]{alpoge-bhargava-shnidman}L. Alp{\"o}ge, M. Bhargava, A. Shnidman, \textit{Integers expressible as the sum of two rational cubes}, with an appendix by A. Burungale and C. Skinner, arXiv:2210.10730 (2022).
    \bibitem[AL]{atkin-lehner}A. O. L. Atkin and J. Lehner, \textit{Hecke operators on $\Gamma_0(N)$}, Math. Ann. 185 (1970), 134-160.
    \bibitem[AS]{akbas-singerman}M. Akba\c{s} and D. Singerman, \textit{The normalizer of $\Gamma_0(N)$ in $\mathrm{PSL}_2(\R)$}, Glasgow Math. J. 32 (1990), 317-327.
    \bibitem[BF]{burungale-flach}A. Burungale and M. Flach, \textit{The conjecture of Birch and Swinnerton-Dyer for certain elliptic curves with complex multiplication}, Camb. J. Math. 12 (2024), no. 2, 357–415.
    \bibitem[Cow]{coward}D. R. Coward, \textit{Some sums of two rational cubes}, Q. J. Math. 51 (2000), no. 4, 451–464.
    \bibitem[CST1]{G-Z formula} L. Cai, J. Shu, and Y. Tian, \textit{Explicit Gross-Zagier and Waldspurger formulae}, Algebra \& Number Theory 8 (2014), no. 10, 2523–2572.
    \bibitem[CST2]{cube sum} L. Cai, J. Shu, and Y. Tian, \textit{Cube sum problem and an explicit Gross-Zagier formula}, Am. J. Math. 139 (2017), no. 3, 785–816.
    \bibitem[DV1]{mock heegner} S. Dasgupta and J. Voight, \textit{Sylvester’s problem and mock Heegner points}, Proc. Amer. Math. Soc. 146 (2018), no. 8, 3257–3273.
    \bibitem[DV2]{survey article} S. Dasgupta and J. Voight, \textit{Heegner points and Sylvester’s conjecture}, Arithmetic
    geometry, Clay Math. Proc., vol. 8, Amer. Math. Soc. Providence, RI, 2009, 91–102.
    \bibitem[Elk]{elkies} N. D. Elkies, \textit{Explicit modular towers},  arXiv preprint math/0103107 (2001).
    \bibitem[GJT]{gjt} E. Gonz{\'a}lez-Jim{\'e}nez and J. M. Tornero, \textit{Torsion of rational elliptic curves over quadratic fields}, Rev. R. Acad. Cienc. Exactas Fís. Nat. Ser. A Mat. RACSAM 108 (2014), no. 2, 923–934.  
    \bibitem[GZ]{gross-zagier} B. H. Gross and D. B. Zagier, \textit{Heegner points and derivatives of $L$-series}, Invent. Math. 84 (1986), 225–320.
    \bibitem[HSY]{hu-shu-yin} Y. Hu, J. Shu, and H. Yin, \textit{An explicit Gross–Zagier formula related to the Sylvester conjecture}, Trans. Am. Math. Soc. 372 (2019), no 10, 6905–6925.
    \bibitem[JMS1]{jha-majumdar-sury}S. Jha, D. Majumdar, and B. Sury, \textit{Binary cubic forms and rational cube sum problem}, Proceedings of the American Mathematical Society 153 (2025), no. 11, 4657-4668.
    \bibitem[JMS2]{cube sum jnt}S. Jha, D. Majumdar, and P. Shingavekar, \textit{$3$-Selmer groups, ideal class groups and the cube sum problem}, Journal of Number Theory 277 (2025), 165-200.
    \bibitem[Kob]{kobayashi} S. Kobayashi, \textit{The p-adic Gross-Zagier formula for elliptic curves at supersingular primes}, Invent. Math. 191 (2013), no. 3, 527–629.
    \bibitem[Kol]{kolyvagin}V. A. Kolyvagin, \textit{Euler systems}, The Grothendieck Festschrift, Vol. II, Progr. Math., vol. 87, Birkh\"auser Boston, Boston, MA, 1990, 435–483.
    \bibitem[Lie]{lieman} D. B. Lieman, \textit{Nonvanishing of L-series associated to cubic twists of elliptic curves}, Ann. of Math. (2) 140 (1994), no. 1, 81–108.
    \bibitem[Lig]{ligozat}G{\'e}rard Ligozat, \textit{Courbes modulaires de genre $1$}, Soci{\'e}t{\'e} math{\'e}matique de France 45 (1975), 5-80.
    \bibitem[LLT]{li-liu-tian} Y. Li, Y. Liu, and Y. Tian, \textit{On the Birch and Swinnerton-Dyer conjecture for CM elliptic curves over $\Q$}, arXiv:1605.01481 (2016).
    \bibitem[Mil]{milne}J. S. Milne, \textit{Arithmetic duality theorems}, 2006. \url{https://www.jmilne.org/math/Books/ADTnot.pdf}
    \bibitem[MS1]{majumdar-sury}D. Majumdar and B. Sury, \textit{Cyclic cubic extensions of $\Q$}, Int. J. Number Theory 18 (2022), no. 1, 1929–1955.
    \bibitem[MS2]{majumdar-shingavekar}D. Majumdar and P. Shingavekar, \textit{Cube sum problem for integers having exactly two distinct prime factors}, Proceedings-Mathematical Sciences 133 (2023), no. 2, p. 43.
    \bibitem[Neu]{neukrich} J. Neukirch, \textit{Algebraic number theory}, Springer Science \& Business Media 322 (2013).
    \bibitem[PR]{perrin-riou}  B. Perrin-Riou, \textit{Points de Heegner et d{\'e}riv{\'e}es de fonctions L p-adiques}, Invent. Math. 89(3) (1987), 455–510.
    \bibitem[Rub]{rubin}K. Rubin, \textit{Tate-Shafarevich groups and $L$-functions of elliptic curves with complex multiplication}. Invent. Math. 89 (1987), no. 3, 527–559.
    \bibitem[Sat]{satge}P. Satg{\'e}, \textit{Groupes de Selmer et corps cubiques}, J. Number Theory 23 (1986), no. 3, 294–317.
    \bibitem[Sel]{selmer} E. S. Selmer, \textit{The Diophantine equation $ax^3 + by^3 + cz^3 = 0$}, Acta Math. 87 (1951), 203–362.
    \bibitem[Shi]{shimura} G. Shimura, \textit{Introduction to the arithmetic theory of automorphic functions}, Publications of the Mathematical Society of Japan, vol. 11, Princeton University Press, Princeton, NJ, 1994.
    \bibitem[Sil]{silverman}J. H. Silverman, \textit{The arithmetic of elliptic curves}, Second edition. GTM Springer,  106 (2009), pp. 513.
    \bibitem[SY]{shu-yin} J. Shu and H. Yin, \textit{Cube sums of the forms $3p$ and $3p^2$ II }, Math. Ann. 385 (3-4) (2023), 1037–1060.
    \bibitem[Syl]{sylvester}J. J. Sylvester, \textit{On certain ternary cubic-form equations}, Am. J. Math. 2 (1879), no. 4, 357–393.
    \bibitem[Yin]{yin} Hongbo Yin, \textit{On the $8$ case of the Sylvester conjecture}, Trans. Am. Math. Soc. 375(2022), no. 4, 2705-2728.
   
    
    
    
\end{thebibliography}
\end{document}